\documentclass[11pt]{amsart}
\usepackage{latexsym}
\usepackage{fullpage}
\usepackage{amssymb}
\usepackage{amscd}
\usepackage{float}
\usepackage{graphicx}
\usepackage{amsfonts}
\usepackage{pb-diagram}
\usepackage{amsmath,amscd}
\usepackage{caption}
\usepackage{subcaption}
\usepackage{xcolor}

\newtheorem{thm}{Theorem}

\newtheorem{lemma}{Lemma}
\newtheorem{prop}{Proposition}
\newtheorem{defn}{Definition}
\newtheorem{remark}{Remark}

\makeatletter
\@namedef{subjclassname@2020}{\textup{2020} Mathematics Subject Classification}
\makeatother

\begin{document}

\title[Tied Links in various topological settings]
  {Tied Links in various topological settings}

\author{Ioannis Diamantis}
\address{China Agricultural University,
International College Beijing, No.17 Qinghua East Road, Haidian District,
Beijing, {100083}, P. R. China.}
\email{ioannis.diamantis@hotmail.com}

\keywords{tied links, handlebody, solid torus, lens spaces, parting, mixed links, mixed braids, tied mixed braids, 3-manifolds, combing.}
\subjclass[2020]{57K10, 57K12, 57K14, 57K35, 57K45, 57K99, 20F36, 20F38, 20C08}

\setcounter{section}{-1}

\date{}

\begin{abstract}
Tied links in $S^3$ were introduced by Aicardi and Juyumaya as standard links in $S^3$ equipped with some non-embedded arcs, called {\it ties}, joining some components of the link. Tied links in the Solid Torus were then naturally generalized by Flores. In this paper we study this new class of links in other topological settings. More precisely, we study tied links in the lens spaces $L(p,1)$, in handlebodies of genus $g$, and in the complement of the $g$-component unlink. We introduce the tied braid monoids $TM_{g, n}$ by combining the algebraic mixed braid groups defined by Lambropoulou and the tied braid monoid, and we formulate and prove analogues of the Alexander and the Markov theorems for tied links in the 3-manifolds mentioned above. We also present an $L$-move braid equivalence for tied braids and we discuss further research related to tied links in knot complements and c.c.o. 3-manifolds. The theory of tied links has potential use in some aspects of molecular biology.
\end{abstract}

\maketitle

\section{Introduction}\label{intro}

Tied links were introduced in \cite{AJ1} as a generalization of links in $S^3$ forming a new class of knotted objects. They are classical links equipped with {\it ties}, i.e. non-embedded arcs joining some components of the link. Tied links are obtained by considering the closure of tied braids, which appear from a diagrammatic interpretation of the defining generators of the {\it algebra of braids and ties} introduced and studied in \cite{Ju, AJ2}. These algebras are related to Coxeter groups of type A. In \cite{F}, the author generalizes the concept of tied links in the Solid Torus in the context of Coxeter groups of type B. In this paper we generalize the concept of tied links in handlebodies of genus $g$, $H_g$, and in other 3-manifolds, such as the lens spaces $L(p,1)$ and the complement of the $g$-component unlink and we present the analogues of the Alexander and the Markov theorems for tied links in these 3-manifolds. It is worth mentioning that tied links have potential applications to molecular biology and in particular to the topology of proteins, where ties represent the chemical bonds in a molecule, in an analogous way as bonded knotoids (see \cite{GGLDSK}).

\smallbreak

We consider $H_g$ to be $S^3\backslash \{{\rm open\ tubular\ neighborhood\ of}\ I_g \}$, where $I_g$ denotes the point-wise fixed identity braid on $g$ indefinitely extended strands meeting at the point at infinity. Thus, $H_g$ may be represented in $S^3$ by the braid $I_g$ (for an illustration see Figure~\ref{h3}). Similarly, $\hat{I}_g$, the standard closure of the identity braid on $g$-strands, represents the complement of the $g$-component unlink in $S^3$, denoted by $S^3\backslash \hat{I}_g$. 

\begin{figure}[ht]
\begin{center}
\includegraphics[width=2.1in]{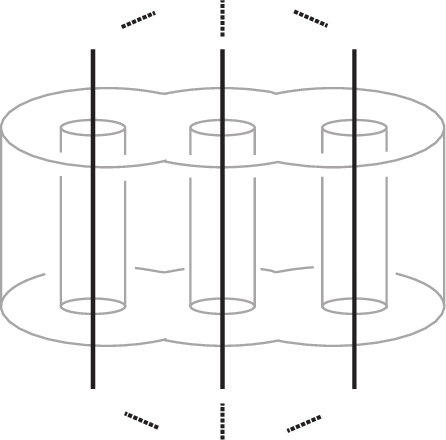}
\end{center}
\caption{A handlebody of genus 3.}
\label{h3}
\end{figure}

As explained in \cite{LR1, LR2, La1, OL, DL1}, an oriented link $L$ in $H_g$ can be represented by an oriented \textit{mixed link} in $S^{3}$, that is, a link in $S^{3}$ consisting of the fixed part $I_g$ and the moving part $L$ that links with $I_g$ (see Figure~\ref{kHg} for an example of a mixed link in $H_3$). A \textit{mixed link diagram} is a diagram $I_g\cup \widetilde{L}$ of $I_g\cup L$ on the plane of $I_g$, where this plane is equipped with the top-to-bottom direction of $I_g$. Similarly, a link in $S^3\backslash \hat{I}_g$ can be seen as a mixed link in $S^3$ as illustrated in Figure~\ref{kUg} for the case of $S^3\backslash \hat{I}_3$.

\begin{figure}[ht]
  \begin{subfigure}[b]{0.35\textwidth}
    \includegraphics[width=\textwidth]{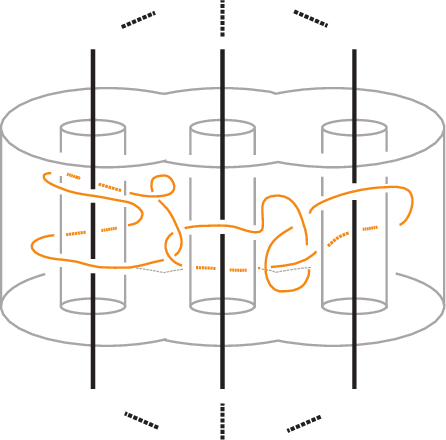}
    \caption{A mixed link in $H_3$.}
    \label{kHg}
  \end{subfigure}
  \begin{subfigure}[b]{0.48\textwidth}
    \includegraphics[width=\textwidth]{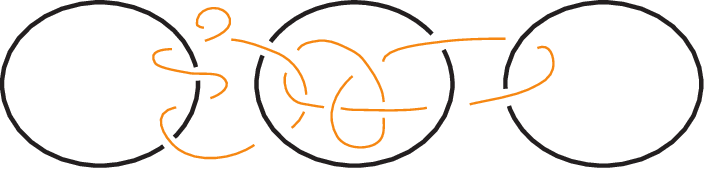}
    \caption{A mixed link in $S^3\backslash \hat{I}_3$.}
    \label{kUg}
  \end{subfigure}
\end{figure}

\smallbreak

For the lens spaces $L(p, 1)$, it is a well known result that these spaces can be obtained from $S^3$ by surgery on the unknot with surgery coefficient $p$. Surgery along the unknot can be realized by considering first the complementary solid torus and then attaching to it a solid torus according to some homeomorphism on the boundary. Thus, a link in $L(p, 1)$ can be visualized as a mixed link in $S^3$, where the fixed part of the mixed link is just $\hat{I}_1$, the unknot, that represents the complementary solid torus in $S^3$.

\bigbreak

In this paper we study isotopy of tied (oriented) links in the above mentioned 3-manifolds, $M$, as well as tied braid equivalence via the concept of tied mixed links, tied mixed braids and tied mixed braid monoids. The paper is organized as follows: In \S~\ref{basics} we recall the knot theory of $M$, emphasizing on the analogue of the Alexander and the (geometric) Markov theorems. We also recall the algebraic mixed braids $B_{g, n}$ in order to present the algebraic version for braid equivalence in $M$. In \S~\ref{sectl} we present some basic results on tied links in $S^3$ (\cite{AJ1}) and ST (\cite{F}) that are necessary for the rest of the paper. In particular, we present the tied braid monoid of type A and of type B, we recall some properties and relations that ties satisfy, which lead to the analogue of the Alexander and the Markov theorems for tied links in $S^3$ (Theorems~\ref{alextls3} \& \ref{marktls3}) and ST (Theorems~\ref{alextlst} \& \ref{marktlst}). Moreover, with the use of the $L$-moves, that we introduce for tied braids, we present $L$-move braid equivalences for tied braids in $S^3$ and in ST. In \S~\ref{seclens} we introduce the concept of tied links in the lens spaces $L(p,1)$, we define tie isotopy by introducing braid band moves on tied links, called tied braid band moves, and we translate tie isotopy on the level of tied mixed braids (Theorem~\ref{markovlens}). Then, in \S~\ref{secM} we study tied links in $H_g$ and $S^3\backslash \hat{I}_g$ by introducing the tied algebraic mixed braid monoids $TM_{g, n}$ (Definition~\ref{tmbmg}), the generalized fixed ties (see Figure~\ref{gt1}) and the properties they satisfy. Finally, in \S~\ref{tlknotc} we discuss the concept of tied links in knot complements by taking into consideration the technique of combing and how ties are affected by this technique (Figure~\ref{combft}). The study of tied links in knot complements is the first step toward the study of tied links in c.c.o. 3-manifolds obtained from $S^3$ by surgery along a knot and that is the subject of a sequel paper.

\section{Knot theory in $H_g$ \& $S^3\backslash \hat{I}_g$}\label{basics}

\subsection{Mixed links and isotopy in $H_g$ and $S^3\backslash \hat{I}_g$}\label{mlisH}

From now on, we will represent the handlebody $H_g$ in $S^3$ by the identity braid $I_g$ and a link $L$ in $H_g$ as a mixed link $I_g \bigcup L$ in $S^3$. Similarly, the complement of the g-component unlink will be represented by $\hat{I}_g$, the (standard) closure of the identity braid $I_g$, and a link $L^{\prime}$ in $S^3\backslash \hat{I}_g$ by a mixed link $\hat{I}_g \bigcup L^{\prime}$ in $S^3$.

\begin{defn}\rm
Two oriented links in $H_g$ or in $S^3\backslash \hat{I}_g$ are {\it isotopic} if and only if any two corresponding mixed link diagrams of theirs in $S^3$ differ by a finite sequence of planar $\Delta$-moves, the three Reidemeister moves and the extended Reidemeister moves, that involve the moving and the fixed part of the mixed link (see Figure~\ref{extreid}).
\end{defn}

\begin{figure}[ht]
\begin{center}
\includegraphics[width=4.3in]{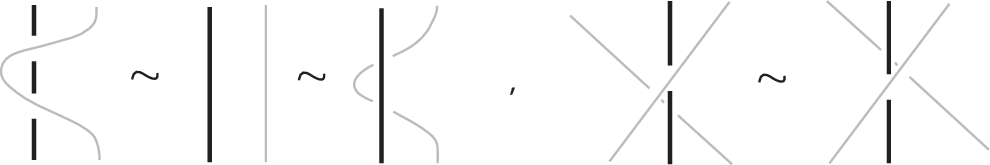}
\end{center}
\caption{Extended Reidemeister moves.}
\label{extreid}
\end{figure}

\subsection{Mixed braids and equivalence in $H_g$ and $S^3\backslash \hat{I}_g$}

Let $M$ denote $H_g$ or $S^3\backslash \hat{I}_g$. In order to translate isotopy for links in $M$ into braid equivalence, we first need to introduce the notion of geometric mixed braids. 

\begin{defn}\rm
A \textit{geometric mixed braid} related to $M$ and to a link $L$ in $M$ is an element of the group $B_{g+n}$, where $g$ strands form the fixed
braid $I_g$ representing $M$ and $n$ strands form the {\it moving subbraid\/} $\beta$ representing the link $L$ in $M$. For an illustration see the middle part of Figure~\ref{mixedbraidsLmoves}. Moreover, a diagram of a geometric mixed braid is a braid diagram projected on the plane of $I_g$.
\end{defn}

The main difference between a geometric mixed braid in $H_g$ and a geometric mixed braid in $S^3\backslash \hat{I}_g$ is the {\it closing} operation that is used to obtain mixed links from geometric mixed braids. In the case of $S^3\backslash \hat{I}_g$, the closing operation is defined in the usual sense (as in the case of classical braids in $S^3$), while in the case of $H_g$ the strands of the fixed part $I_g$ do not participate in the closure operation. In particular, the closure operation in $H_g$ is defined as follows:

\begin{defn}\rm
The {\it closure} $C(I_g \cup B)$ of a geometric mixed braid $I_g \cup B$ in $H_g$ is realized by joining each pair of corresponding endpoints of the moving part by a vertical segment, either over or under the rest of the braid and according to the label attached to these endpoints (see Figure~\ref{bHg}).
\end{defn}

\begin{figure}[ht]
\begin{center}
\includegraphics[width=5.6in]{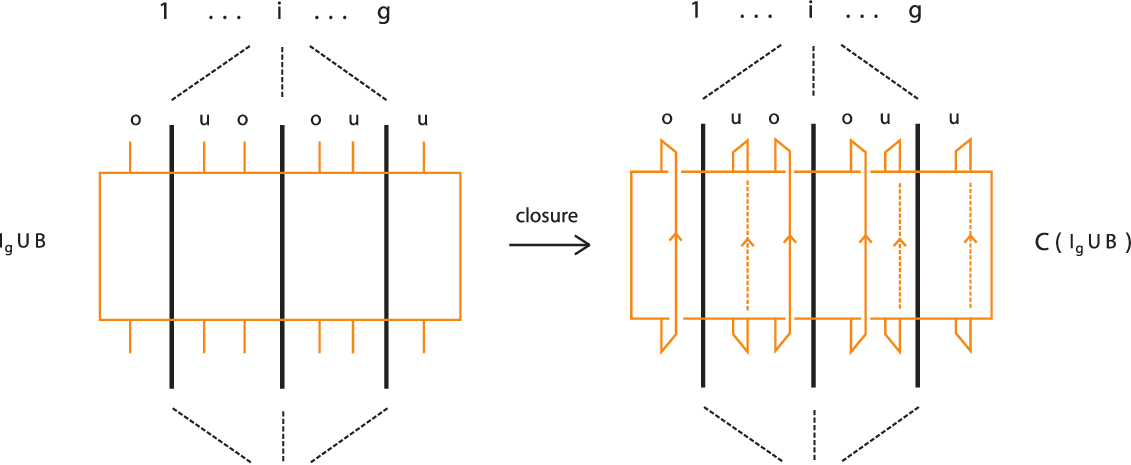}
\end{center}
\caption{A geometric mixed braid in $H_g$ and its closure.}
\label{bHg}
\end{figure}

\begin{remark}\rm
It is crucial to note that different labels on the endpoints of a geometric mixed braid will yield non isotopic links in $H_g$ in general. For more details the reader is referred to \cite{OL}. 
\end{remark}

\subsection{The analogue of the Alexander Theorem for knots in $M$}\label{alsec}

Before we proceed with the analogue of the Alexander theorem for knots in $M$, we first need to introduce the notion of $L$-moves:

\begin{defn}\label{lmdefn}\rm
An {\it $L$-move} on a geometric mixed braid $I_g \bigcup \beta$, consists in cutting an arc of the moving subbraid $\beta$ open and pulling the upper cutpoint downward and the lower  upward, so as to create a new pair of braid strands with corresponding endpoints
(on the vertical line of the cutpoint), and such that both strands  cross entirely  {\it over} or {\it under}
with the rest of the braid. Stretching the new strands over will give rise to an {\it $L_o$-move\/} and under to an {\it  $L_u$-move\/} as shown in Figure~\ref{mixedbraidsLmoves}. Note that an $L$-move does not touch the fixed subbraid $I_g$.
\end{defn}

\begin{figure}[ht]
\begin{center}
\includegraphics[width=5.75in]{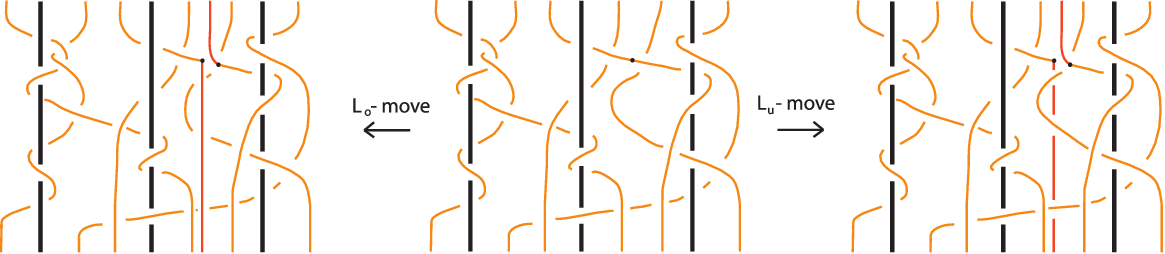}
\end{center}
\caption{A geometric mixed braid and the two types of $L$-moves. }
\label{mixedbraidsLmoves}
\end{figure}

We now recall the braiding algorithm from \cite{LR1}, the main idea of which is to keep the arcs of the oriented link diagram that go downwards with respect to the height function unaffected and replace arcs that go upwards with braid strands. These arcs are called {\it opposite arcs} and obviously these arcs do not belong to the fixed subbraid $I_g$.

\smallbreak

We first review the braiding process for mixed links in $H_g$, which can be summarized as follows:

\bigbreak

\begin{itemize}
\item We chose a base-point and we run along the diagram of the mixed link according to its orientation.
\smallbreak
\item When/If we run along an opposite arc, we subdivide it into smaller arcs, each containing crossings of one type only as shown in Figure~\ref{upa}. These arcs are called {\it up-arcs}.
\smallbreak

\begin{figure}[ht]
\begin{center}
\includegraphics[width=4.7in]{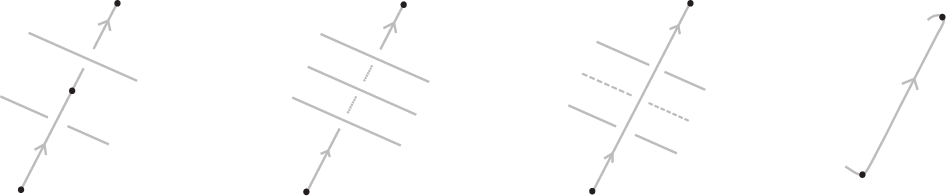}
\end{center}
\caption{Up-arcs.}
\label{upa}
\end{figure}

\item We now label every up-arc with an ``o''or a ``u'', according to the crossings it contains. If it contains no crossings, then the choice is arbitrary.
\smallbreak
\item We perform an $L_o$-move on all up-arcs which were labeled with an ``o'' and an $L_u$-move on all up-arcs which were labeled with an ``u'' (see Figure~\ref{ahg} by ignoring the red springs).
\smallbreak

\item The result is a geometric mixed braid whose closure is isotopic to the initial mixed link.
\end{itemize}

\smallbreak

\noindent Note that the braiding process does not involve the fixed part of the mixed link.

\bigbreak

For mixed links in $S^3\backslash \hat{I}_g$ the idea is the same but before applying the braiding algorithm, one has to guarantee that $I_g$ will remain fixed during the process. In particular, consider a mixed link $\hat{I}_g \bigcup L$. $\hat{I}_g$ in $S^3$ represents the closure of the identity braid $I_g$, and thus, it can be viewed as the union of $I_g$ with an arc at infinity, that resembles the closure of $I_g$. By general position arguments, the link $L$ can be isotoped to the complement of a tubular neighborhood of this arc in $S^3$, say $W$, which contains $I_g$. The braiding algorithm on $W$ will not affect $I_g$ and it will braid $L$. 

\smallbreak

Using the braiding algorithm, the following theorem is proved in \cite{LR1} (see also \cite{OL} for a detailed proof for the case of $H_g$):

\begin{thm}[{\bf The analogue of the Alexander theorem for} $M$]
An oriented mixed link $I_g \cup L$ in $M$ may be braided to a geometric mixed braid $I_g \cup B$, the closure of which is isotopic to $I_g \cup L$.
\end{thm}

\subsection{Braid equivalence for knots in $M$}\label{ms3}

In \cite{LR1} it is shown that braid equivalence in $S^3$ as well as mixed braid equivalence in $M$ are generated only by the $L$-moves. So we have the following:

\begin{thm}[{\bf Geometric braid equivalence for} $M$, Theorem~3 \cite{OL} \& Theorem~4.2 \cite{LR1}] \label{geommarkov}
Two oriented links in $M$ are isotopic if and only if any two corresponding geometric mixed braids in $S^3$ differ by $L$-moves that do not touch the fixed part $I_g$ of the mixed braid.
\end{thm}

In order now to construct invariants for knots in $M$ using the braid approach, we need to move toward an algebraic statement of Theorem~\ref{geommarkov}. Toward that end, we introduce the notion of {\it algebraic mixed braids}.

\begin{defn}\rm
An \textit{algebraic mixed braid} is a mixed braid on $g+n$ strands such that the first $g$ strands are fixed and form the identity braid on $g$ strands and the next $n$ strands are moving strands and represent a link in the manifold $M$.
\end{defn}

Clearly, an algebraic mixed braid is a special case of a geometric mixed braid. Moreover, geometric mixed braids can be turned to algebraic mixed braids (see \cite{OL} Lemma~1) using the technique of {\it parting}. {\it Parting} a geometric mixed braid $I_g \bigcup\beta$ on $g+n$ strands means to separate its endpoints into two different sets, the first $g$ belonging to the subbraid $I_g$ and the last $n$ to $\beta$, and so that the resulting braids have isotopic closures. This can be realized by pulling each pair of corresponding moving strands to the right and {\it over\/} or {\it under\/} each  strand of $I_g$ that lies on their right according to its label. We start from the rightmost pair respecting the position of the endpoints (for an abstract illustration see Figure~\ref{partingandcombing}). The result of parting is an {\it algebraic mixed braid}. 

\smallbreak

The set of all algebraic mixed braids on $g+n$ strands forms a subgroup of $B_{g+n}$, denoted $B_{g,n}$, called the {\it mixed braid group}. The mixed braid group $B_{g,n}$ has been introduced and studied in \cite{La1} and it is shown that it has the following presentation:

\begin{equation} \label{B}
B_{g,n} = \left< \begin{array}{ll}  \begin{array}{l}
a_1, \ldots, a_g,  \\
\sigma_1, \ldots ,\sigma_{n-1}  \\
\end{array} &
\left|
\begin{array}{l} \sigma_k \sigma_j=\sigma_j \sigma_k, \ \ |k-j|>1   \\
\sigma_k \sigma_{k+1} \sigma_k = \sigma_{k+1} \sigma_k \sigma_{k+1}, \ \  1 \leq k \leq n-1  \\
{a_i} \sigma_k = \sigma_k {a_i}, \ \ k \geq 2, \   1 \leq i \leq g    \\
 {a_i} \sigma_1 {a_i} \sigma_1 = \sigma_1 {a_i} \sigma_1 {a_i}, \ \ 1 \leq i \leq g  \\
 {a_i} (\sigma_1 {a_r} {\sigma^{-1}_1}) =  (\sigma_1 {a_r} {\sigma^{-1}_1})  {a_i}, \ \ r < i
\end{array} \right.  \end{array} \right>,
\end{equation}

\noindent where the {\it loop generators} $a_i$ and the {\it braiding generators} $\sigma_j$ are as illustrated in Figure~\ref{Loops and Crossings}.

\begin{figure}[ht]
\begin{center}
\includegraphics[width=5.2in]{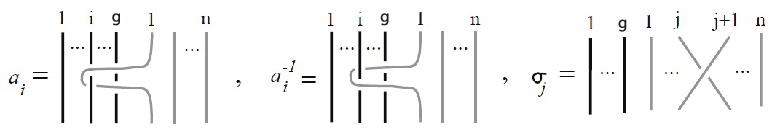}
\end{center}
\caption{ The loop generators $a_i, \ {a^{-1}_i}$ and the braiding generators $\sigma_j$ of $B_{g,n}$. }
\label{Loops and Crossings}
\end{figure}

The group $B_{g,n}$ embeds naturally into the group $B_{g,n+1}$. We shall denote $B_{g,\, \infty}:=\underset{n=1}{\overset{\infty}{\bigcup}}\, B_{g,n}$, the disjoint union of all braid groups $B_{g, n}$. We now state the algebraic version of Theorem~\ref{geommarkov}.

\begin{thm}[{\bf Algebraic mixed braid equivalence for} $M$, Theorem~5 \cite{LR2}] \label{algcco}
Two oriented links in  $M$ are isotopic if and only if any two corresponding algebraic mixed braid representatives in $B_{g,\, \infty}$ differ by a finite sequence of the following moves:

\[
\begin{array}{llccll}
1. & {\rm Markov\ move:} & {\beta}_1 {\beta}_2  & \sim & {\beta}_1{\sigma^{\pm 1}_n}{\beta}_2, & {\rm for} \ {\beta}_1, {\beta}_2  \in B_{g,n},\\
&&&&&\\
2. & {\rm Markov\ conjugation:} & \beta & \sim & {\sigma^{\pm 1}_j} \beta {\sigma^{\mp 1}_j}, & {\rm for} \ \beta, \sigma_j \in B_{g,n},\\
&&&&&\\
3. & {\rm Loop\ conjugation\ {\bf only\ for}}\ S^3\backslash \hat{I}_g{\rm :} & \beta  & \sim &  {a^{\mp 1}_i} \beta {a^{\pm 1}_i}, & {\rm for} \ \beta, a_i \in B_{g,n}.\\
\end{array}
\]
\end{thm}

\begin{remark}\label{lmhg}\rm
Relations (1) and (2) in Theorem~\ref{algcco} may be replaced by {\it algebraic} $L$-moves, that is, geometric $L$-moves between elements in $B_{g,\, \infty}$ that preserve the group structure of the algebraic mixed braids. Note also that some loop conjugations in $H_g$ are allowed and that in \S~5 \cite{OL} it is shown how these conjugations can be translated in terms of relations (1) and (2) of Theorem~\ref{algcco}. 

\end{remark}

\smallbreak

Following \cite{Jo}, Theorem~\ref{algcco} is the first step toward the construction of link invariants in $M$. This method has been successfully applied in the case of HOMFLYPT type invariants for knots in the Solid Torus in \cite{La2, La3, DL2} and work toward the construction of such invariants for knots in lens spaces $L(p,1)$ has been done in \cite{DL3, DL4, DLP, D3}. For the case of Kauffman bracket type invariants via braids the reader is referred to \cite{D1} for knots in the Solid Torus, and to \cite{D2} for knots in the genus 2 handlebody.

\section{Tied Links in $S^3$ \& ST}\label{sectl}

In this section we introduce the concepts of tied links and their diagrams following \cite{AJ1} for the case of $S^3$ and \cite{F} for the case of the Solid Torus ST. We also present an $L$-move braid equivalence for tied braids in $S^3$ and for tied braids in ST.

\subsection{Tied Links in $S^3$}

As explained in \cite{AJ1}, a tied link is a link whose set of components is subdivided into classes. More precisely:

\begin{defn}\label{tls3}\rm
A {\it tied (oriented) link} $L(P)$ on $n$ components is a set $L$ of $n$ disjoint smooth (oriented) closed curves embedded in $S^3$, and a set $P$ of ties, i.e., unordered pairs of points $(p_r, p_s)$ of such curves between which there is an arc called a {\it tie}. If $P=\emptyset$, then $L$ is a classical link in $S^3$.
\end{defn}

Note that ties are depicted as red \color{red} springs \color{black}or \color{red}line segments \color{black}connecting pairs of points lying on the curves and that ties are not embedded arcs, i.e. arcs can cross through ties. It is also worth mentioning that the set of ties on a tied link defines a partition of the set of components of the links. This leads to the following definition:

\begin{defn}\rm
A tie is called {\it essential} if we obtain a different partition in the set of components by removing it.
\end{defn}

Moreover, a {\it tied link diagram} is defined as a diagram of a link provided with ties which are depicted as springs connecting pairs of points lying on the curves.

\smallbreak

We introduce now the notion of isotopy between tied links, called {\it tie isotopy}. It is natural to define tie isotopy between tied links as ambient isotopy between links (ignoring the ties), and also impose that the set of ties in those links define the same partition of the set of components of the links. More precisely:

\begin{defn}\label{istls3}\rm
Two oriented tied links $L(P)$ and $L^{\prime}(P^{\prime})$ are {\it tie isotopic} if:
\smallbreak
\begin{itemize}
\item the links $L$ and $L^{\prime}$ in $S^3$ are ambient isotopic and
\smallbreak
\item the sets $P$ and $P^{\prime}$ define the same partition of the set of components of $L$ and $L^{\prime}$.
\end{itemize}
\end{defn}

Our aim now is to study tied links via {\it tied braids}, that is, usual braids in $S^3$ equipped with ties. We first need to relate tied links to tied braids (by the analogue of the Alexander theorem) and then understand how tie isotopy can be translated on the level of tied braids (by the analogue of the Markov theorem). Toward that end, we introduce the {\it monoid of tied braids} which plays the role of the braid group for tied links.

\begin{defn}\label{montls3}\rm
The tied braid monoid $TM_{n}$ is the monoid generated by $\sigma_1, \ldots, \sigma_{n-1}$, the usual braid generators of $B_{n}$, and the generators $\eta_1, \ldots, \eta_{n-1}$, called ties (see Figure~\ref{gs3}), satisfying the braid relations of $B_{n}$ together with the following relations:

\[
\begin{array}{lcl}
\eta_i \eta_j & = & \eta_j \eta_i, \qquad {\rm for\ all}\ i, j\\
&&\\
\eta_i \sigma_i & = & \sigma_i \eta_i, \qquad {\rm for\ all}\ i\\
&&\\
\eta_i \sigma_j & = & \sigma_j \eta_i, \qquad {\rm for\ all}\ |i-j|>1\\
&&\\
\eta_i \sigma_j \sigma_i^{\epsilon} & = & \sigma_j \sigma_i^{\epsilon} \eta_j, \quad {\rm for}\ |i-j|=1\ {\rm and\ where}\ \epsilon\, =\, \pm 1\\
&&\\
\eta_i \eta_j \sigma_i & = & \eta_j \sigma_i \eta_j\ =\ \sigma_i \eta_i \eta_j, \quad {\rm for}\ |i-j|=1\\
&&\\
\eta_i^2 & = & \eta_i, \qquad \ \ {\rm for\ all}\ i\\
\end{array}
\]
\end{defn}

\begin{figure}[ht]
\begin{center}
\includegraphics[width=3.2in]{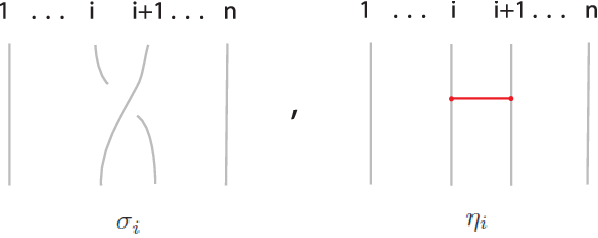}
\end{center}
\caption{Generators of $TM_n$.}
\label{gs3}
\end{figure}

\begin{defn}\rm
The {\it generalized tie} $\eta_{i, j}$, joining the $i^{th}$ with the $j^{th}$ strand is defined as:

\begin{equation}\label{genties}
\eta_{i, j}\ :=\ \sigma_i\, \ldots\, \sigma_{j-2}\, \eta_{j-1}\, \sigma_{j-2}^{-1}\, \ldots,\, \sigma_i^{-1}
\end{equation}

Set $\eta_{i, i}:=1$ and define also the {\it length} of a generalized tie $\eta_{i, j}$ as $l(\eta_{i, j})=|i-j|$ and define $l(\eta_i)=1$.
\end{defn}

\smallbreak

Generalized ties $\eta_{i, j}$ are {\it transparent} with respect to all strands between the $i^{th}$ and $j^{th}$ strands, that is, they can be drawn no matter if in front or behind these strands. Moreover, the generalized ties are provided with {\it elasticity}, i.e. each generalized tie can be transformed by a second Reidemeister move in which the tie is stretched and represented as a spring. The generalized ties satisfy the following relations (\cite{AJ1}):

\begin{equation}\label{eqf}
\begin{array}{crclcrcl}
(i) & \sigma_i\, \eta_{i, j} & = & \eta_{i+1}\, \sigma_i, & (ii) & \sigma_j\, \eta_{i, j} & = & \eta_{i, j+1}\, \sigma_i,\\
&&&&&&&\\
(iii) & \sigma_{i-1}\, \eta_{i, j} & = & \eta_{i-1, j}\, \sigma_{i-1}, & (iv) & \sigma_{j-1}\, \eta_{i, j} & = & \eta_{i, j-1}\, \sigma_{j-1}, \\
&&&&&&&\\
(v) & \eta_{i, k}\, \eta_{k, m} & = & \eta_{i, k}\, \eta_{i, m} & = & \eta_{k, m}\, \eta_{i, m},& & 1\leq i,\ k, m\leq n\\
\end{array}
\end{equation}

Using these properties and Equations~\ref{eqf}, in \cite{AJ1} it is shown that:

\begin{prop}\rm \label{propaj}
\begin{itemize}
\item[i.] Mobility property: In any tied braid all ties can be moved to the bottom (or to the top), that is, $\alpha\, \sim\, \beta\, \gamma$, where $\alpha\in TM_n, \beta \in B_n$ and $\gamma$ the set of generalized ties.
\smallbreak
\item[ii.] Any set of generalized ties in $TM_n$ defines an equivalence relation on the set of $n$ strands.
\smallbreak
\item[iii.] Let $\alpha_1=\beta_1\, \gamma_1,\ \alpha_2=\beta_2\, \gamma_2$ be two tied braids in $TM_n$. Then, $\alpha_1\, \sim\, \alpha_2$, if and only if $\beta_1\, \sim\, \beta_2$ in $B_n$ and $\gamma_1, \gamma_2$ define the same partition of the set of the strands. 
\end{itemize}
\end{prop}

By considering now $TM_n \subset TM_{n+1}$, we can consider the inductive limit $TM_{\infty}$. Using the generalized ties and the fact that ties can be transferred freely, in \cite{AJ1} it is shown that Lambropoulou's braiding algorithm can also be applied for tied links. In particular we have the following results:

\begin{thm}[{\bf The analogue of the Alexander theorem for tied links in} $S^3$] \label{alextls3}
Every oriented tied link can be obtained by closing a tied braid.
\end{thm}

\begin{thm}[{\bf The analogue of the Markov theorem for tied links}] \label{marktls3}
Two tied braids have tie isotopic closures if and only if one can  obtained from the other by a finite sequence of the following moves:

\[
\begin{array}{lllcll}
1^{\prime}. & {\rm Markov\ Conjugation:} &  \alpha\, \beta & \sim & \beta\, \alpha, & {\rm for\ all}\ \alpha,\, \beta \in TM_n,\\
&&&&&\\
2^{\prime}. & {\rm Markov\ Stabilization:} &  \alpha & \sim & \alpha\, \sigma_n^{\pm 1}, & {\rm for\ all}\ \alpha \in TM_n,\\
&&&&&\\
3^{\prime}. & {\rm Ties:} &  \alpha & \sim & \alpha\, \eta_{i, j}, & {\rm for\ all}\ \alpha \in TB_n\ {\rm such\ that}\\
\end{array}
\]

\noindent $s_{\alpha}(i)=j$, {\rm where} $s_{\alpha}$ {\rm denotes the permutation associated to the braid obtained from} $\alpha$ {\rm by forgetting its ties.}
\end{thm}

\begin{remark}\rm
\begin{itemize}
\item[] It is worth mentioning that in \cite{AJ1, AJ3, AJ4} the authors define Jones type invariants for tied links in $S^3$, by showing first that the bt-algebra (the algebraic counterpart of the braid group for tied links) supports a Markov trace.
\smallbreak
\item[] Moreover, the theory of {\it tied pseudo links}, i.e. tied links with some missing crossing information, is introduced and studied in \cite{D4}. In particular, the author formulates and proves the analogues of the Alexander and Markov theorems for tied pseudo links in $S^3$. For the case of pseudo links in $H_g$ without ties, the reader is referred to \cite{D7}.
\end{itemize}
\end{remark}

\bigbreak
\bigbreak

\noindent {\bf $L$-moves and tied links in $S^3$.} We now introduce the notion of $L$-moves for tied links in $S^3$ and our intention is to present an $L$-move braid equivalence theorem for tied braids (recall Theorem~\ref{algcco} and Remark~\ref{lmhg}). An $L$-move on a tied braid in $S^3$ is defined as in Definition~\ref{lmdefn} for geometric mixed braids, by ignoring the fixed parts of the geometric mixed braids and by observing that due to the transparency property of the ties, the new pair of braid strands that appear after the performance of an $L$-move may cross above or below all ties of the tied braid (for an illustration of the performance of an $L_o$-move on a tied braid see Figure~\ref{lmtb}). 

\begin{figure}[ht]
\begin{center}
\includegraphics[width=4.8in]{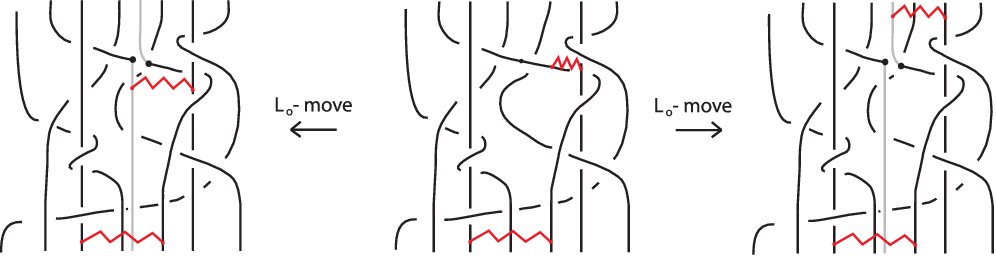}
\end{center}
\caption{An $L_o$-move on a tied braid.}
\label{lmtb}
\end{figure}

Moreover, the new braid strands that appear after the performance of an $L$-move will not affect the partition induced on the tied link by the set of ties (see Figure~\ref{lmtb}). Finally, as shown in \cite{LR1}, the $L$-moves can realize conjugation, while stabilization moves are special cases of $L$-moves. Thus, relations $1^{\prime}$ and $2^{\prime}$ of Theorem~\ref{marktls3} may be replaced by $L$-moves. This leads to the following theorem:

\begin{thm}[{\bf $L$-move equivalence for tied braids}]\label{lmtbs3}
Two tied braids have isotopic closures if and only if one can obtained from the other by a finite sequence of the following
moves:

\[
\begin{array}{llcll}
{\rm L-moves} &   &  &  & \\
&&&&\\
{\rm Ties:} &  \alpha & \sim & \alpha\, \eta_{i, j}, & {\rm for\ all}\ \alpha \in TM_n\ {\rm such\ that}\\
\end{array}
\]
\end{thm}

\subsection{Tied Links in ST}

We now generalize the concept of tied links in $S^3$ to tied links in the Solid Torus ST following \cite{F}. We consider ST to be the complement of a solid torus in $S^3$. As explained in \S~\ref{mlisH}, an oriented link $L$ in ST can be represented by an oriented mixed link in $S^{3}$, that is a link in $S^{3}$ consisting of the unknotted fixed part $\widehat{I}$ representing the complementary solid torus in $S^3$ and the moving part $L$ that links with $\widehat{I}$. Similarly to Definition~\ref{tls3}, a tied (oriented) link in ST is a mixed link equipped with ties that may also involve the fixed part of the link. We will call such links {\it tied mixed links}. For an illustration see Figure~\ref{tmlink}.

\begin{figure}[ht]
\begin{center}
\includegraphics[width=2.4in]{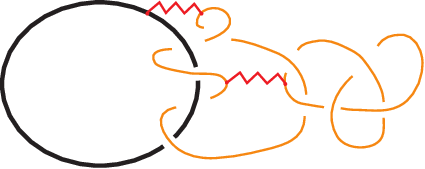}
\end{center}
\caption{A tied mixed link in $S^3$.}
\label{tmlink}
\end{figure}

As in the case of tied links in $S^3$, the set of ties induces a partition on the set of the components of the link $L$, where two components of $L$ belong to the same class if they are connected by a tie. Moreover, isotopy for tied links in ST is translated to mixed link isotopy, i.e. the classical Reidemeister moves for the moving part of the tied mixed link, the extended Reidemeister moves that involve both the fixed and the moving part of the link, and also $P$ and $P^{\prime}$ define the same set of partition on the components of the mixed link, as in the case of tied links in $S^3$ (see Definition~\ref{istls3}).

\smallbreak

We now present the {\it monoid of tied braids of type B} which plays the role of the braid group for tied mixed links in ST.

\begin{defn}\label{montlst}\rm
The {\it tied braid monoid of type B}, $TM_{1, n}$ is the monoid generated by $\sigma_1, \ldots, \sigma_{n-1}$ and $a_1$, the usual braid generators of the mixed braid group $B_{1, n}$, the generators $\eta_1, \ldots, \eta_{n-1}$, called ties and illustrated in Figure~\ref{gs3} and $\phi_1$, called {\it fixed tie} (see Figure~\ref{ftie}), satisfying the braid relations of $B_{1, n}$ together with relations in Definition~\ref{montls3}, and also:

\[
\begin{array}{rcll}
\phi_1^2 & = & \phi_1,& \\
&&&\\
a_1 \eta_i & = & \eta_i a_1, & {\rm for\ all}\ i\\
&&&\\
a_1 \phi_1 & = & \phi_1 a_1,& \\
&&&\\
\phi_1 \eta_1 & = & \eta_i \phi_1,& {\rm for\ all}\ i\\
&&&\\
\phi_1 \sigma_1 & = & \sigma_i \phi_1, & {\rm for\ all}\ 2\leq i \leq n-1\\
&&&\\
\sigma_{i-1} \ldots \sigma_1 \phi_1 \sigma_1^{-1} \ldots \sigma_{i-1}^{-1} & = & \sigma_{i-1}^{-1} \ldots \sigma_1^{-1} \phi_1 \sigma_1 \ldots \sigma_{i-1},& {\rm for\ all}\ 2\leq i \leq n\\
&&&\\
\phi_1 \eta_1 & = & \phi_1 \sigma_1 \phi_1 \sigma_1^{-1}\ =\ \sigma_1 \phi_1 \sigma_1^{-1}\phi_1&\\
\end{array}
\]
\end{defn}

\begin{figure}[ht]
\begin{center}
\includegraphics[width=1.2in]{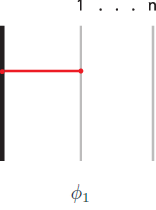}
\end{center}
\caption{The fixed tie $\phi_1$.}
\label{ftie}
\end{figure}

\begin{remark}\rm
Comparing Definitions \ref{montls3} and \ref{montlst}, we have that $TM_n \subset TM_n^B$ and moreover, $TM_n^B \subset TM_{n+1}^B$, which allows us to define the inductive limit $TM_{\infty}^B$.
\end{remark}

\begin{defn}\rm
The {\it generalized fixed tie} $f_j$, joining the fixed strand with the $j^{th}$ strand, is defined as follows:

\[
f_j\ :=\ \sigma_j\, \ldots\, \sigma_1\, \phi_{1}\, \sigma_{1}^{-1}\, \ldots,\, \sigma_j^{-1}
\]
\end{defn}

The generalized ties and the generalized fixed ties satisfy the following relations (Eq.~20--22 \cite{F}):

\begin{equation}\label{eqf1}
\begin{array}{llcll}
(i) & f_j\, \eta_{i} & = & \eta_{i}\, f_j, & {\rm for\ all}\ 1\leq i \leq n-1\ \&\ 1\leq j\leq n,\\
&&&&\\
(ii) & f_{j}\, \sigma_{i} & = & \sigma_{i}\, f_{s_i(j)}, & {\rm where}\ s_i\ {\rm is\ the\ transposition}\ (i\ i+1), \\
&&&&\\
(iii) & \eta_{i, j}\, f_i & = & f_j\, f_i\ =\ f_j\, \eta_{i, j} & {\rm for\ all}\ 1\leq i\neq j\leq n.\\
\end{array}
\end{equation}

Using Eq.~\ref{eqf} and \ref{eqf1} in \cite{F}, Proposition~\ref{propaj} is generalized to tied braids in ST. In particular:

\begin{prop}\rm\label{propmf}
\begin{itemize}
\item[i.] Mobility property: In any tied braid in ST, all ties can be moved to the bottom (or to the top), that is, $\alpha\, \sim\, \beta\, \gamma$, where $\alpha\in TM_{1, n}, \beta \in B_{1, n}$ and $\gamma$ the set of generalized and generalized fixed ties.
\smallbreak
\item[ii.] Any set of generalized ties in $TM_{1, n}$ defines an equivalence relation on the set of $1+n$ strands.
\end{itemize}
\end{prop}

We are now in position to state the analogues of the Alexander and the Markov Theorems for tied links in ST.

\begin{thm}[{\bf The analogue of the Alexander theorem for tied links in ST}] \label{alextlst}
Every oriented tied link can be obtained by closing a tied braid.
\end{thm}

An alternative proof of Theorem~\ref{alextlst} is based on the fact that the steps of the braiding algorithm mentioned in \S~\ref{alsec}, involve moves that resemble the $L$-moves and the fact that by Definition~\ref{lmdefn}, the performance of $L$-moves on a tied mixed braid will leave the partition on the set of the components of the link unchanged.

We now proceed by stating the analogue of the Markov theorem for tied links in ST.

\begin{thm}[{\bf The analogue of the Markov Theorem for tied links in ST}] \label{marktlst}
Two tied braids in ST have tie isotopic closures if and only if one can  obtained from the other by a finite sequence of the following moves:

\[
\begin{array}{lllcll}
1^{\prime \prime}. & {\rm Markov\ Conjugation:} &  \alpha\, \beta & \sim & \beta\, \alpha, & {\rm for\ all}\ \alpha,\, \beta \in TM_n,\\
&&&&&\\
2^{\prime \prime}. & {\rm Markov\ Stabilization:} &  \alpha & \sim & \alpha\, \sigma_n^{\pm 1}, & {\rm for\ all}\ \alpha \in TM_n,\\
&&&&&\\
3^{\prime \prime}. & {\rm Ties:} &  \alpha & \sim & \alpha\, \eta_{i, j}, & {\rm for\ all}\ \alpha \in TM_n\ :\ s_{\alpha}(i)=j,\\
&&&&&\\
4^{\prime \prime}. & {\rm Fixed\ Ties:} & \alpha & \sim & f_j\, \alpha, & {\rm whenever} \ s_{\alpha}(i)=j\ \&\ \alpha\ {\rm contains\ a}\ f_i.\\
\end{array}
\]
\end{thm}

\bigbreak
\bigbreak

\noindent {\bf $L$-moves and tied links in ST.} $L$-moves on tied braids in ST are defined as in the classical case for mixed braids in $H_1:={\rm ST}$, due to the transparency of the ties (for an illustration see Figure~\ref{lmtb2}). 

\begin{figure}[ht]
\begin{center}
\includegraphics[width=3.2in]{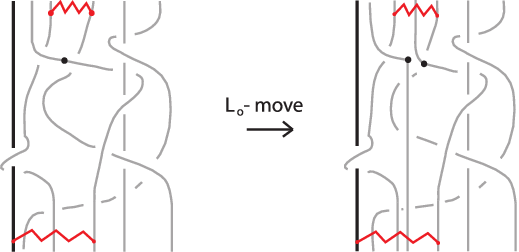}
\end{center}
\caption{An $L$-move on a tied braid in ST.}
\label{lmtb2}
\end{figure}

As noted before, the $L$-moves may replace conjugation and stabilization moves for mixed braids (recall Theorem~\ref{algcco} and Remark~\ref{lmhg}). Moreover, the performance of an $L$-move can not affect the partition induced on the tied mixed braids by the ties. We have the following:

\begin{thm}[{\bf $L$-move equivalence for tied braids in ST}]\label{lmtbs4}
Two tied braids in ST have isotopic closures if and only if any two mixed tied braids representatives of their differ by a finite sequence of the following moves:

\[
\begin{array}{lllcll}
{\rm L-moves} &   &  &  & \\
&&&&\\
{\rm Ties:} &  \alpha & \sim & \alpha\, \eta_{i, j}, & {\rm for\ all}\ \alpha \in TM_{1, n}\ :\ s_{\alpha}(i)=j,\\
&&&&\\
{\rm Fixed\ Ties:} & \alpha & \sim & f_j\, \alpha, & {\rm whenever} \ s_{\alpha}(i)=j\ \&\ \alpha\ {\rm contains\ a}\ f_i.\\
\end{array}
\]
\end{thm}

\section{Tied Links in $L(p,1)$}\label{seclens}

As mentioned in \S~\ref{intro}, we consider the lens spaces $L(p,1)$ as a result of surgery in $S^3$ along the unknot with surgery coefficient $p$. As explained in \cite{LR1, LR2, DL1, DLP, DL3, D3}, isotopy in $L(p, 1)$ can be viewed as isotopy in ST together with the two types of {\it band moves} in $S^3$, which reflect the surgery description of the manifold (see Figure~\ref{bmov}).

\begin{figure}[ht]
\begin{center}
\includegraphics[width=5.4in]{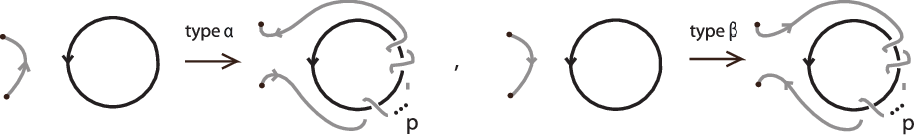}
\end{center}
\caption{The two types of band moves.}
\label{bmov}
\end{figure}

Moreover, since the endpoints of the ties can freely move along the arcs of the mixed link, the performance of band moves will not affect the partition on the components of the link by the ties (see Figure~\ref{bmov1}).

\begin{figure}[ht]
\begin{center}
\includegraphics[width=5.2in]{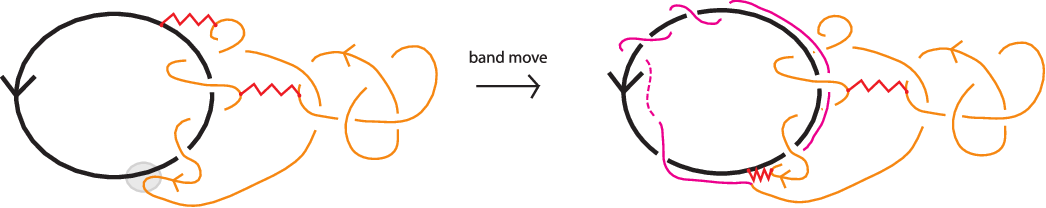}
\end{center}
\caption{A band move on a tied link.}
\label{bmov1}
\end{figure}

Thus, tie isotopy in $L(p,1)$ can be viewed as tie isotopy in ST together with the {\it band moves} in $S^3$. In \cite{DL2} it is also shown that one type of band moves can be obtained from the other type of band moves together with isotopy in ST. We will consider only the type $a$ band moves as illustrated in Figure~\ref{bmov} and thus, tie isotopy between oriented tied links in $L(p,1)$ is reflected in $S^3$ by means of the following theorem:

\begin{thm}[{\bf Reidemeister's Theorem for tied links in} $L(p,1)$]\label{reidlens}
Two oriented tied links in $L(p,1)$ are isotopic if and only if two corresponding mixed tied link diagrams of theirs differ by tie isotopy in ST together with a finite sequence of the type $a$ band moves.
\end{thm}

In order now to translate tie isotopy for links in $L(p,1)$ into braid equivalence, we first use Theorem~\ref{alextlst} to braid the tied links in ST. Then, we define a {\it tied braid band move}, abbreviated as {\it t-bbm}, to be a move between tied mixed braids, which is a band move between their closures. It starts with a little band oriented downward, which, before sliding along a surgery strand, gets one twist {\it positive\/} or {\it negative\/} (see Figure~\ref{bbmfig}).

\begin{figure}[ht]
\begin{center}
\includegraphics[width=2.2in]{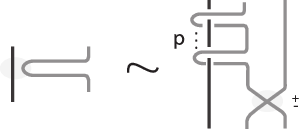}
\end{center}
\caption{The two types of braid band moves.}
\label{bbmfig}
\end{figure}

\begin{lemma}\rm \label{bbmlem}
The braid band moves may always be considered to take place on the {\it last} moving strand and at the {\it bottom} of the tied mixed braid.
\end{lemma}

\begin{proof}
We first observe that the bbm's can be performed on a any moving strand of the tied mixed braid, but using conjugation (i.e. tie isotopy in ST), we can always consider the bbm to take place on the last moving strand of the tied mixed braid (see Lemma~1 \cite{DLP} and consider the fact that the endpoints of the ties may move freely). 

\smallbreak

In order to prove now that a bbm can always be assumed to take place at the bottom of a tied mixed braid, we follow the proof of Lemma~5 in \cite{LR2}. Consider a little band of a mixed braid $I\bigcup \beta_1$ approaching the fixed part $I$. Using tie isotopy in ST, we can pull this little band along $I$ in order to bring it to the bottom of the mixed braid. The result is not a mixed braid, so we braid it in order to obtain a mixed braid again, say $I\bigcup \beta_2$. Obviously, $I\bigcup \beta_1$ is tie isotopic to $I\bigcup \beta_2$, since the shifting of the band to the bottom of the mixed braid can be realized by isotopy in ST and the partition of the set of components is unchanged. Perform now a bbm on $I\bigcup \beta_1$ and $I\bigcup \beta_2$ and call the resulting tied mixed braids $I\bigcup b_1$ and $I\bigcup b_2$. Obviously, $I\bigcup b_1\, \sim\, I\bigcup b_2$ by the same tie isotopies as $I\bigcup \beta_1 \, \sim\, I\bigcup \beta_2$. 

\smallbreak

Moreover, as shown in \cite{DL1}, the braid band moves will have the following algebraic expression:

\begin{equation}\label{algbbm}
\beta \sim \beta^{\prime} {a^{\prime}_n}^p \sigma_n^{\pm 1}
\end{equation}

\noindent {\rm where} ${a^{\prime}_n}=\sigma_n \ldots \sigma_1 a_1 \sigma_1^{-1} \ldots \sigma_n^{-1}$ and $\beta^{\prime}$ is the word $\beta$ with the substitutions:

$$a_1^{\pm 1} \longleftrightarrow (\sigma_1^{-1} \ldots \sigma_{n-1}^{-1} \cdot  \sigma_n^2\cdot \sigma_{n-1}\ldots \sigma_{1}a_1)^{\pm 1} $$

Moreover, the performance of a t-bbm does not affect the partition of the components of the tied mixed braid since by simple calculations we have:

\[
\begin{array}{lcll}
f_i\, {a^{\prime}_n}^p\, \sigma_n^{\pm 1} & \sim &  {a^{\prime}_n}^p\, \sigma_n^{\pm 1}\, f_{i}, & {\rm for}\ 1\leq i \leq n-2\\
&&&\\
f_n\, {a^{\prime}_n}^p\, \sigma_n^{\pm 1} & \sim &  {a^{\prime}_n}^p\, \sigma_n^{\pm 1}\, f_{n+1}, &\\
&&&\\
\eta_i\, {a^{\prime}_n}^p\, \sigma_n^{\pm 1} & \sim &  {a^{\prime}_n}^p\, \sigma_n^{\pm 1}\, \eta_{i}, & {\rm for}\ 1\leq i \leq n-2\\
&&&\\
\eta_{n-1}\, {a^{\prime}_n}^p\, \sigma_n^{\pm 1} & \sim &  {a^{\prime}_n}^p\, \sigma_n^{\pm 1}\, \eta_{n-1, n+1}. &\\
\end{array}
\]

The proof is now concluded.
\end{proof}

From Lemma~\ref{bbmlem} and the discussion above it follows that tie isotopy in $L(p,1)$ is translated on the level of tied mixed braids by means of the following theorem:

\begin{thm}[{\bf The analogue of the Markov Theorem in} $L(p,1)$] \label{markovlens}
 Let $L_{1} ,L_{2}$ be two oriented tied links in $L(p,1)$ and let $I\cup \beta_{1} ,{\rm \; }I\cup \beta_{2}$ be two corresponding tied mixed braids in $S^{3}$. Then $L_{1}$ is isotopic to $L_{2}$ in $L(p,1)$ if and only if $I\cup \beta_{1}$ is equivalent to $I\cup \beta_{2}$ in $\mathop{\cup }\limits_{n=1}^{\infty } TM_{1,n}$ by the following moves:
\[
\begin{array}{lllcll}
1^{\prime \prime \prime}. & {\rm Markov\ Conjugation:} &  \alpha\, \beta & \sim & \beta\, \alpha, & {\rm for\ all}\ \alpha,\, \beta \in TM_{1,n},\\
&&&&&\\
2^{\prime \prime \prime}. & {\rm Markov\ Stabilization:} &  \alpha & \sim & \alpha\, \sigma_n^{\pm 1}, & {\rm for\ all}\ \alpha \in TM_{1,n},\\
&&&&&\\
3^{\prime \prime \prime}. & {\rm Ties:} &  \alpha & \sim & \alpha\, \eta_{i, j}, & {\rm for\ all}\ \alpha \in TM_{1,n}\ :\ s_{\alpha}(i)=j,\\
&&&&&\\
4^{\prime \prime \prime}. & {\rm Fixed\ Ties:} & \alpha & \sim & f_j\, \alpha, & {\rm whenever} \ s_{\alpha}(i)=j\ \&\ \alpha\ {\rm contains\ a}\ f_i.\\
&&&&&\\
5^{\prime \prime \prime}. & {\rm t-bbm's:} & \alpha & \sim & \alpha^{\prime}{a^{\prime}_n}^p \sigma_n^{\pm 1}, & \alpha^{\prime}\in TM_{1, n+1}\quad (Eq.~\ref{algbbm}).\\
\end{array}
\]
\end{thm}

The properties that ties satisfy and Theorem~\ref{lmtbs4}, suggest that moves $1^{\prime \prime \prime}$ and $2^{\prime \prime \prime}$ of Theorem~\ref{markovlens} may be replaced by the $L$-moves (see also \cite{DL1, LR1}) leading to an $L$-move equivalence for tied braids in $L(p, 1)$:

\begin{thm}[{\bf $L$-move equivalence for tied braids in $L(p, 1)$}]\label{lmtbs5}
Two tied braids in $L(p, 1)$ have isotopic closures if and only if any two mixed tied braids representatives of their differ by a finite sequence of the following moves:

\[
\begin{array}{lllcll}
{\rm L-moves} &   &  &  & \\
&&&&\\
{\rm Ties:} &  \alpha & \sim & \alpha\, \eta_{i, j}, & {\rm for\ all}\ \alpha \in TM_{1,n}\ :\ s_{\alpha}(i)=j,\\
&&&&\\
{\rm Fixed\ Ties:} & \alpha & \sim & f_j\, \alpha, & {\rm whenever} \ s_{\alpha}(i)=j\ \&\ \alpha\ {\rm contains\ a}\ f_i.\\
&&&&&\\
{\rm t-bbm's:} & \alpha & \sim & \alpha^{\prime}{a^{\prime}_n}^p \sigma_n^{\pm 1}, & \alpha^{\prime}\in TM_{1, n+1}\quad (Eq.~\ref{algbbm}).\\
\end{array}
\]
\end{thm}

\section{Tied Links in $M$}\label{secM}

In this section we introduce the concept of tied links in $M\, =\, H_g$ or $S^3\backslash \hat{I}_g$ and their diagrams. Tied links in $M$ naturally generalize classical links in $M$ and tied links in $S^3$ and ST. The difference between tied links in $M$ and tied links in ST lies in the fixed part $I_g$ of the tied mixed link and also that in the set $P$, the set of ties, various {\it fixed} ties are allowed. For an illustration of a tied link in $H_3$, the handlebody of genus 3, see Figure~\ref{tknotHg}, while for an illustration of a tied link in $S^3\backslash \hat{I}_3$, the complement of the 3-component unlink see Figure~\ref{tknotug}. A diagram of $L(P)$, is a mixed link diagram of $L$ in $S^3$ projected on the plane of $I_g$ (equipped with the top-to-bottom direction), equipped with ties connecting pairs of points in the set of ties $P$.

\begin{figure}[ht]
  \begin{subfigure}[b]{0.4\textwidth}
    \includegraphics[width=0.9\textwidth]{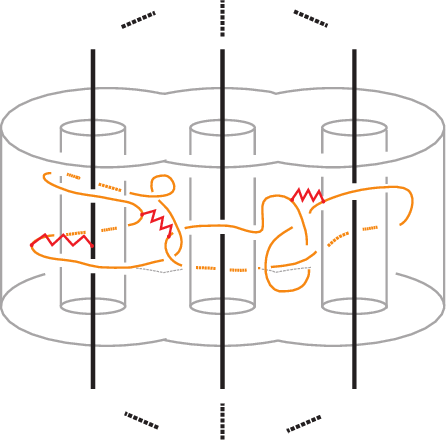}
    \caption{A tied link in $H_3$.}
    \label{tknotHg}
  \end{subfigure}
  \begin{subfigure}[b]{0.5\textwidth}
    \includegraphics[width=1\textwidth]{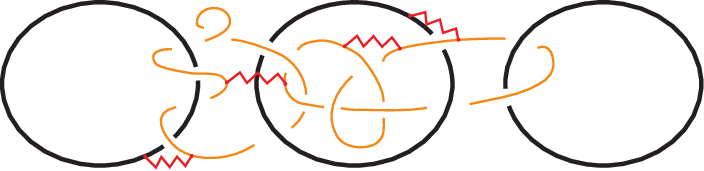}
    \caption{A tied knot in $S^3\backslash \hat{I}_3$.}
    \label{tknotug}
  \end{subfigure}
\end{figure}

Similarly to $S^3$ and ST, two oriented tied links $L(P), L^{\prime}(P^{\prime})$ in $M$, are tie isotopic, denoted by $L(P)\, \sim\, L^{\prime}(P^{\prime})$, if:
\begin{itemize}
\item[$\bullet$] $L\, \sim\, L^{\prime}$ in $M$ and
\smallbreak
\item[$\bullet$] $P$ and $P^{\prime}$ define the same partition in the set of components of $L$ and $L^{\prime}$ respectively.
\end{itemize}

\subsection{The tied braid monoids $TM_{g, n}$}

We now introduce the tied mixed braid monoids $TM_{g, n}$ in order to obtain the analogues of the Alexander and the Markov theorems for tied links in $M$. Denote the {\it fixed tie} joining the $j^{th}$ fixed strand with the $1^{st}$ moving strand by $\phi_j$ (see Figure~\ref{gph}).

\begin{figure}[ht]
\begin{center}
\includegraphics[width=2.2in]{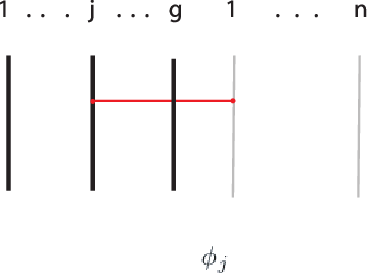}
\end{center}
\caption{The fixed tie $\phi_j$.}
\label{gph}
\end{figure}

\begin{defn}\label{tmbmg}\rm
We define the {\it tied mixed braid monoid} $TM_{g, n}$, as the monoid generated by $\alpha_1, \ldots, \alpha_g, \sigma_1, \ldots, \sigma_{n-1}$, the usual braid generators of $B_{g, n}$, the generators $\phi_1, \ldots, \phi_g$, called {\it fixed ties} and the generators $\eta_1, \ldots, \eta_{n-1}$, called ties, satisfying the relations~(\ref{B}) of $B_{g, n}$ and relations from Definition~\ref{montlst}, together with the following relations:

\[
\begin{array}{rcl}
\alpha_i\, \eta_j & = & \eta_j\, \alpha_i, \qquad {\rm for\ all}\ i, j\\
&&\\
\alpha_i\, \phi_j & = & \phi_j\, \alpha_i, \qquad {\rm for\ all}\ i, j\\
&&\\
\eta_i\, \phi_j & = & \phi_j\, \eta_i, \qquad {\rm for\ all}\ j\ {\rm and}\ i>1,\\
&&\\
\phi_i\, \eta_1 & = & \phi_i\, \sigma_1\, \phi_i\, \sigma_1^{-1}\ =\ \sigma_1\, \phi_i\, \sigma_1^{-1}\, \eta_1\\
&&\\
\phi_i\, \phi_j & = & \phi_j\, \phi_i, \qquad {\rm for\ all}\ i, j\\
&&\\
\phi_j^2 & = & \phi_j, \qquad \ \ \ {\rm for\ all}\ j\\
&&\\
\phi_j\, \sigma_i & = & \sigma_i\, \phi_j, \qquad {\rm for\ all}\ i, j\\
&&\\
\sigma_i \ldots \sigma_1\, \phi_j\, \sigma_1^{-1} \ldots \sigma_i^{-1} & = & \sigma_i^{-1} \ldots \sigma_1^{-1} \, \phi_j\, \sigma_1 \ldots \sigma_i,\, \quad {\rm for\ all}\ i, j\\
\end{array}
\]
\end{defn}

\begin{figure}[ht]
\begin{center}
\includegraphics[width=5.5in]{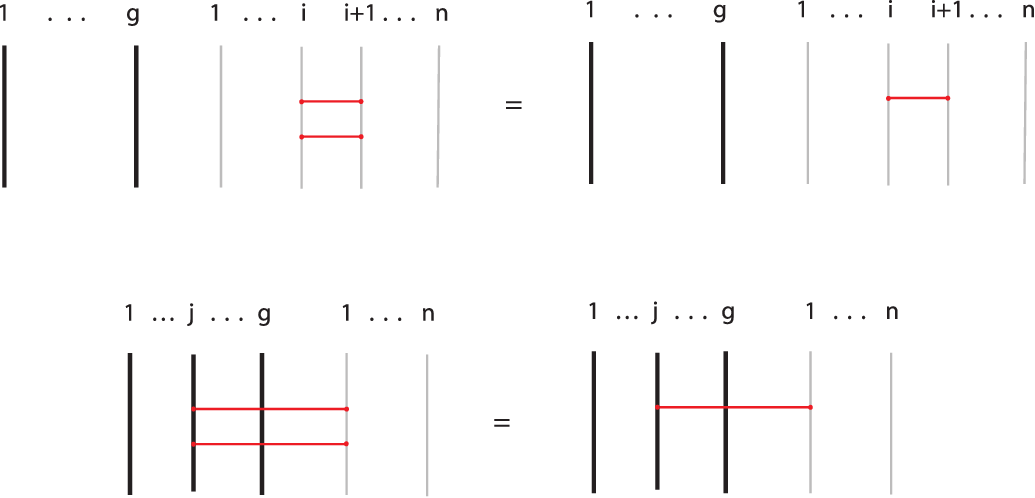}
\end{center}
\caption{The relations $\eta_i^2=\eta_i$ and $\phi_j^2=\phi_j$.}
\label{relt}
\end{figure}

\begin{figure}[ht]
\begin{center}
\includegraphics[width=5.5in]{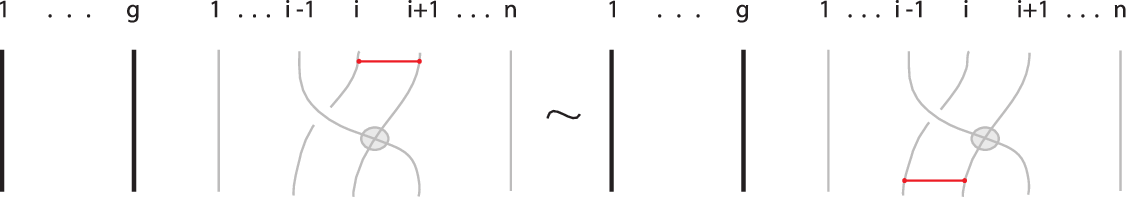}
\end{center}
\caption{Relations in $TM_{g, n}$.}
\label{relt1}
\end{figure}

\begin{figure}[ht]
\begin{center}
\includegraphics[width=4.1in]{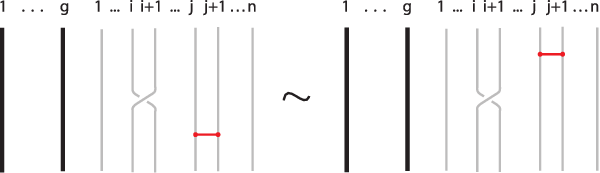}
\end{center}
\caption{The relations $\sigma_i\, \eta_{j}\, =\, \eta_{j}\, \sigma_i$.}
\label{relt2}
\end{figure}

\begin{figure}[ht]
\begin{center}
\includegraphics[width=4.1in]{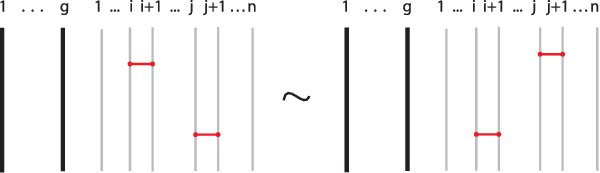}
\end{center}
\caption{The relations $\eta_i\, \eta_{j}\, =\, \eta_{j}\, \eta_i$.}
\label{relt3}
\end{figure}

Note that $TM_{g, n} \subset TM_{g, n+1}$, and thus, $TM_{g,\, \infty}\, :=\, \cup_{n\geq 1}\, TM_{g, n}$, the inductive limit, is well defined.

\subsection{Generalized Ties}

Following the same ideas as in \cite{AJ1, F}, we first need to get the analogue of Propositions~\ref{propaj}, \ref{propmf} for tied links in $M$.

\begin{defn}\rm
Define the {\it generalized fixed tie}, $\phi_{i, j}$, joining the $i^{th}$ fixed strand with the $j^{th}$ moving strand of the tied mixed braid as follows:

\[
\phi_{i, j}\ :=\ \sigma_{j-1}\, \ldots\, \sigma_1\, \phi_{i}\, \sigma_{1}^{-1}\, \ldots,\, \sigma_{j-1}^{-1},\ {\rm for}\ 2\leq j \leq n.
\]
\end{defn}

\begin{figure}[ht]
\begin{center}
\includegraphics[width=5.5in]{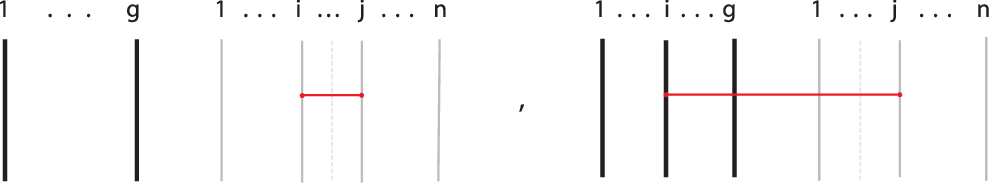}
\end{center}
\caption{The generalized ties $\eta_{i, j}$ \& the generalized fixed ties $\phi_{i, j}$.}
\label{gt1}
\end{figure}

\begin{figure}[ht]
\begin{center}
\includegraphics[width=5.5in]{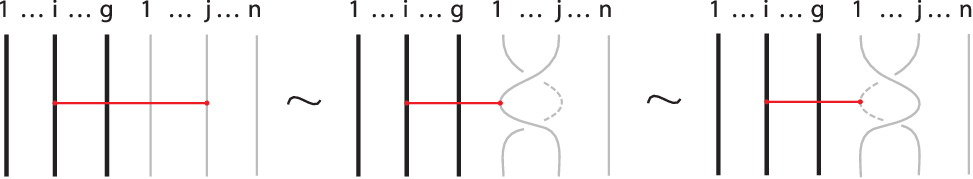}
\end{center}
\caption{Relations between generalized fixed ties.}
\label{gt3}
\end{figure}

Similarly to Equations~\ref{eqf} and \ref{eqf1}, and as a consequence of straightforward calculations using the relations in Definition~\ref{tmbmg}, we obtain the following:

\begin{lemma}\label{gtrel}\rm
The generalized fixed ties and the generalized ties satisfy the following relations:
\[
\begin{array}{crcll}
(i^{\prime}) & \eta_k\, \phi_{i, j} & = & \phi_{i, j}\, \eta_k, & {\rm for\ all}\ i, j, k \\
&&&\\
(ii^{\prime}) & \phi_{i, j}\, \sigma_{k} & = & \sigma_{k}\, \phi_{i, s_k(j)}, & {\rm for\ all}\ i, j, k \\
&&&\\
(iii^{\prime}) &\eta_{i, j}\, \phi_{k, i} & = & \phi_{k, i}\, \phi_{k, j}\ & =\ \, \phi_{k, j}\, \eta_{i, j},\ 1\leq i \, \neq\, j\leq n,\, \forall k. \\
\end{array}
\]
\end{lemma}

Denote by $TM_{g, n}^{\sim}$ the subset of $TM_{g, n}$ generated by the elements $\eta_{i, j}$ and $\phi_{k, m}$, $1\leq i, j \leq n-1$ and $1\leq k \leq g, 1\leq m \leq n$. An immediate consequence of Lemma~\ref{gtrel} and Relations~\ref{eqf} is the following result, which is the analogue of Propositions~\ref{propaj}(i) \& \ref{propmf}(i):

\begin{prop}[Mobility property]\label{mobprop} \rm
Let $\alpha$ be a tied braid in $TM_{g, n}$. Then, $\alpha$ can be written as $\alpha\, =\, \gamma\, \beta$ (or $\alpha\, =\, \beta\, \gamma$), where $\gamma\in B_{g, n}$ and $\beta\in TM_{g, n}^{\sim}$.
\end{prop}


\subsection{The analogues of the Alexander \& the Markov Theorems for $H_g$ \& $S^3\backslash \hat{I}_g$}

We now state and prove the analogues of the Alexander and the Markov Theorems for tied links in $M$.

\begin{thm}[{\bf The analogue of the Alexander theorem for tied links in} $M$] \label{alextlm}
Every oriented tied link can be obtained by closing a tied mixed braid.
\end{thm}

\begin{proof}
The proof is analogous to that of \cite{LR1, OL}, where we have to take care of the ties that are present. More precisely:

\smallbreak

\begin{itemize} 
\item[$H_g$:] We start from a tied mixed link in $H_g$ and before applying the braiding process described in \S~\ref{alsec}, we move the endpoints of the ties along strands and rearrange them by placing them on arcs that go downwards. Recall that the braiding process in $H_g$ does not involve the fixed part $I_g$. Thus, using the transparency property of the ties and the fact that the moves involved in the braiding algorithm (that resemble the $L$-moves; see Figure~\ref{ahg}) do not change the partition on the set of the components of the mixed link, we can guarantee that the braiding algorithm will not affect the partition defined by the ties. The result follows.

\begin{figure}[ht]
\begin{center}
\includegraphics[width=4in]{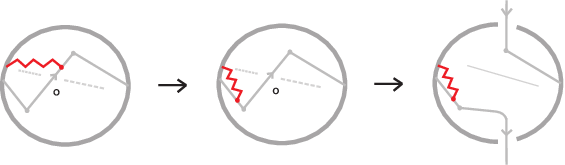}
\end{center}
\caption{Moves involved in the braiding algorithm.}
\label{ahg}
\end{figure}

\bigbreak
\item[$S^3\backslash \hat{I}_g$:] For tied links in $S^3\backslash \hat{I}_g$ the idea is the same, but before applying the braiding process mentioned above, we first need to move the tied link in the complement of an arc at infinity that realizes the closure of $I_g$ in the same way we did for the case of mixed links in $S^3\backslash \hat{I}_g$.
\end{itemize}
\end{proof}

We now proceed in stating and proving the analogue of the Markov theorem for tied links in $M$. For that we need the following definition:

\begin{defn}\label{simeq}\rm
Two tied braids in $TM_{g, n}$ are $\sim_H$-equivalent if one can be obtained from the other by applying a finite sequence of the following moves:
\[
\begin{array}{llclcll}
a. & {\rm Markov\ Conjugation} & : &  \alpha & \sim & \sigma^{\pm 1}\, \alpha\, \sigma^{\mp 1}, & {\rm for\ all}\ \alpha \in TM_{g,n},\\
&&&&&&\\
b. & {\rm Markov\ Stabilization} & : &  \alpha & \sim & \alpha\, \sigma_n^{\pm 1}, & {\rm for\ all}\ \alpha \in TM_{g,n},\\
&&&&&&\\
c. & {\rm Ties} & : &  \alpha & \sim & \alpha\, \eta_{i, j}, & {\rm for\ all}\ \alpha \in TM_{g,n}\ :\ s_{\alpha}(i)=j,\\
&&&&&&\\
d. & {\rm Fixed\ Ties} & : & \alpha & \sim & \phi_{i, j}\, \alpha, & {\rm whenever} \ s_{\alpha}(i)=j\ \&\ \alpha\ {\rm contains\ a}\ \phi_{k, i}.\\
\end{array}
\]

\noindent If moreover two tied braids in $TM_{g, n}$ differ by a finite sequence of the moves (a, b, c, d) together with
\[
\begin{array}{llclcll}
e. & {\rm Loop\ Conjugation} & : &  \alpha & \sim & a_i^{\pm 1}\, \alpha\, a_i^{\mp 1}, & {\rm for\ all}\ \alpha \in TM_{g,n},\\
\end{array}
\]
\noindent then the tied braids are called $\sim_U$-equivalent.
\smallbreak
\end{defn}

\begin{thm}[{\bf The analogue of the Markov Theorem in} $H_g\, \& \, S^3\backslash \hat{I}_g$] \label{markovhg}
Let $\alpha_1, \alpha_2$ be tied braids in $TM_{g, n}$. Then, the links $L_1 = \widehat{\alpha_1}$ and $L_2 = \widehat{\alpha_2}$ are tie isotopic:
\smallbreak

\begin{itemize}
\item[i.] in $H_g$ if and only if $\alpha_1\, \sim_H\, \alpha_2$ and 
\smallbreak
\item[ii.] in $S^3\backslash \hat{I}_g$ if and only if $\alpha_1\, \sim_U\, \alpha_2$. 
\end{itemize}
\end{thm}

\begin{proof}
Let $L_1, L_2$ be two tie isotopic links in $H_g$ (cor. $S^3\backslash \hat{I}_g$) and let $B_1, B_2$ the corresponding tied mixed braids. We prove that $B_1\, \sim_{H}\, B_2$ (cor. $B_1\, \sim_{U}\, B_2$). By Proposition~\ref{mobprop} we have that $B_1=\beta_1\, \gamma_1$ and $B_2=\beta_2\, \gamma_2$, where $\beta_1, \beta_2\in B_{g, n}$ and $\gamma_1, \gamma_2 \in TM_{g, n}^{\sim}$.

\smallbreak

Since the links $L_1, L_2$ are tie isotopic, we have that they are isotopic by ignoring the ties. Thus, we have that $B_1$ is equivalent to $B_2$ (by ignoring the ties again) and we can obtain $B_1$ from $B_2$ via a sequence of Definition~\ref{simeq} (a, b)-moves for $H_g$ and Definition~\ref{simeq} (a, b, e)-moves for $S^3\backslash \hat{I}_g$. Thus, $B_2$ can be written as $\beta_1\, \gamma$. We now deal with ties. By definition, the ties of $B_1$ and $B_2$, define the same partition of the corresponding mixed link components $C$. Thus, it suffices to prove that $\gamma_1$ is related to $\gamma$ via a sequence of Definition~\ref{simeq} (c, d)-moves. We consider generalized ties and fixed generalized ties separately:

\smallbreak

\begin{itemize}
\item[$\eta_{i, j}$:] For the generalized ties, as shown in \cite{AJ1} Theorem~3.7, we have that $\gamma_1\, \lambda=\, \gamma\, \lambda$, where $\lambda\, =\, \underset{s_{\beta_1}(i)}{\prod}{\eta_{i, j}}$. Thus, by Definition~\ref{simeq} (c)-moves, we have that $\gamma_1\, \sim_{H}\, \gamma$ and $\gamma_1\, \sim_U\, \gamma$.
\smallbreak
\item[$\phi_{k, m}$:] Similarly to \cite{F} Theorem~5, for generalized fixed ties, if $\gamma_1$ contains $\phi_{i, j}$, since $\gamma_1$ and $\gamma$ define the same partition in the set of components of $B_1$, $\gamma$ should contain either $\phi_{i, j}$ or $\phi_{i, k}$, for some $k$ in the cycle of $s_{\beta_1}$ that contains $j$, say $c_j$. For a cycle $c_m$ let
\[
\omega_{c_m}\, :=\, \underset{j\in c_m}{\prod}{\phi_{i, j}}\quad {\rm and}\quad \omega\, =\, \underset{\phi_{i, j}\in \gamma_1}{\prod}{\omega_{c_j}}.
\]
We have that $\omega\, \lambda\, \gamma_1\, =\, \omega\, \lambda\, \gamma$, and thus, by Definition~\ref{simeq} (c, d)-moves, we have that $\gamma_1\, \sim_{H}\, \gamma$ and $\gamma_1\, \sim_U\, \gamma$.
\end{itemize}

\end{proof}

\begin{remark}\rm
\begin{itemize}
\item[i.] Note that from the discussion on $L$-moves, it follows that moves ($a$) and ($b$) of Definition~\ref{simeq}, i.e. conjugation and stabilization moves, may be replaced by $L$-moves for tied mixed braids (recall Remark~\ref{lmhg}, Theorems~\ref{lmtbs3}, \ref{lmtbs4}), leading to an $L$-move braid equivalence for tied links in $M$.
\smallbreak
\item[ii.] In \S~6 \cite{AJ}, the authors present a mechanism for constructing {\it tied algebras} from tied braid monoids, which may lead to Jones type invariants for tied links in 3-manifolds. This is the subject of a sequel paper.
\end{itemize}
\end{remark}

\section{On tied links in knot complements}\label{tlknotc}

Tied links can be naturally generalized to knot complements and other c.c.o. 3-manifolds. In order to obtain the algebraic analogue of the Markov theorem for tied links in knot complements $M$, we start from tied mixed links in $S^3$ and using the braiding method that we described for $S^3\backslash \hat{I}_g$ Theorem~\ref{alextlm}, we obtain a tied geometric mixed braid. Note now that the fixed part of the geometric mixed braid does not necessarily consist of the identity braid $I_g$. Let $B$ denote the fixed part. Then, as in the case of $S^3\backslash \hat{I}_g$, we apply the technique of parting in order to separate the strands of the mixed braid into two sets: the fixed part $B$ and the moving part, say $L$. We call the result a parted mixed braid. Then, using the technique of {\it combing}, we separate the braiding of the fixed subbraid $B$ from the braiding of the moving strands using mixed braid isotopy (see Figure~\ref{partingandcombing} for an example of parting and combing a geometric mixed braid).

\begin{figure}[ht]
\begin{center}
\includegraphics[width=5.7in]{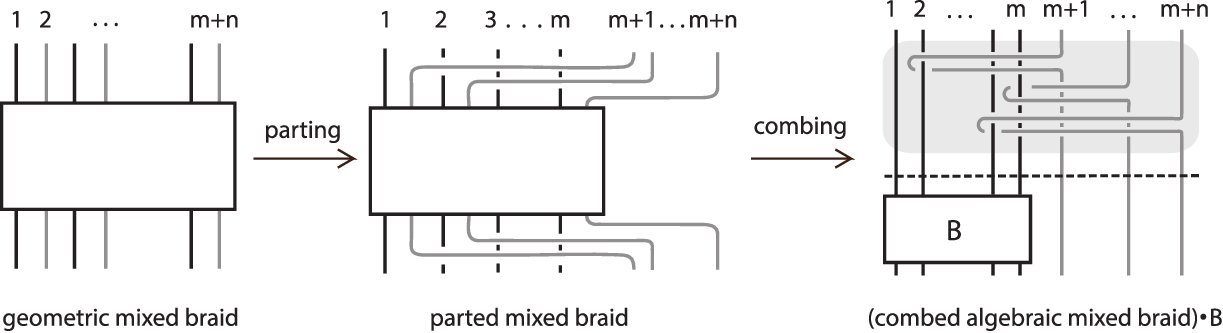}
\end{center}
\caption{ Parting and combing a geometric mixed braid. }
\label{partingandcombing}
\end{figure}

Let now $\Sigma_k$ denote the crossing between the $k^{th}$ and the $(k+1)^{st}$ strand of the fixed subbraid. Then, for all $j=1,\ldots,n-1$ and $k=1,\ldots,m-1$ we have: $\Sigma_k \sigma_j = \sigma_j \Sigma_k$. Thus, the only generating elements of the moving part that are affected by the combing are the loops $a_i$ as illustrated in Figure~\ref{comb}. 

\begin{figure}[ht]
\begin{center}
\includegraphics[width=5.6in]{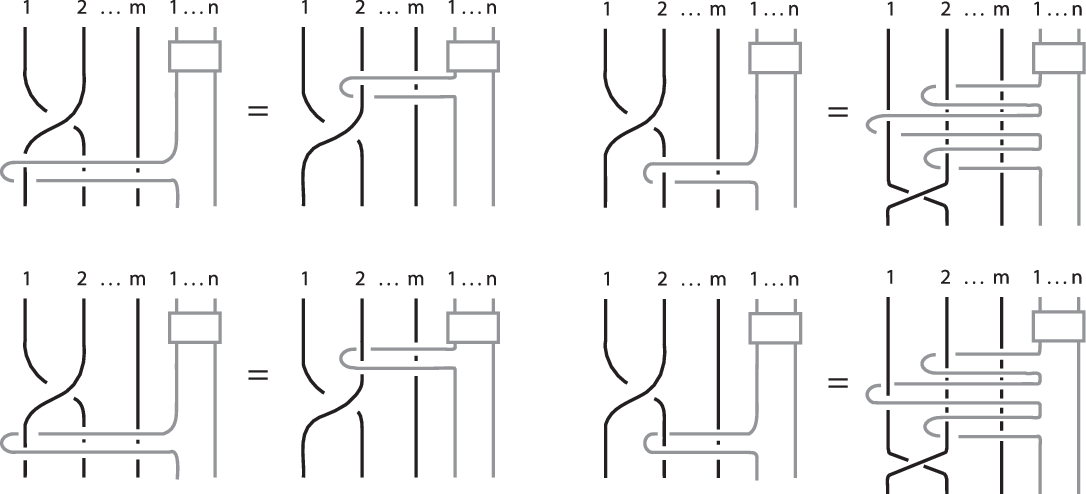}
\end{center}
\caption{ Combing. }
\label{comb}
\end{figure}

In order now to obtain tie braid equivalence for tied links in $M$, we need to take into consideration the following relations (see Figure~\ref{combft}):

\begin{equation}
\Sigma_i\, \phi_{i, m}\ =\ \phi_{i+1, m}\, \Sigma_i,\qquad \Sigma_i\, \phi_{i+1, m}\ =\ \phi_{i, m}\, \Sigma_i
\end{equation}

\begin{figure}[ht]
\begin{center}
\includegraphics[width=4in]{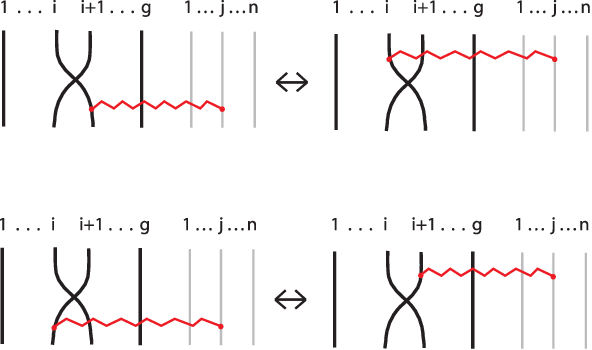}
\end{center}
\caption{ Combing the fixed generalized ties. }
\label{combft}
\end{figure}

\noindent and this is the subject of a sequel paper, i.e. tied links in knot complements, which is the first step toward studying tied links in 3-manifolds obtained from $S^3$ by (rational) surgery along a knot.

\end{document}